\documentclass{amsart}
\usepackage{amssymb,accents}
\usepackage[all]{xy}

\newtheorem{Lem}{Lemma}[section]
\newtheorem{Thm}[Lem]{Theorem}
\newtheorem{Cor}[Lem]{Corollary}
\theoremstyle{definition}
\newtheorem{Def}[Lem]{Definition}
\theoremstyle{remark}
\newtheorem{Fact}[Lem]{Fact}
\newtheorem{Facts}[Lem]{Facts}
\newtheorem{Rem}[Lem]{Remark}

\DeclareMathOperator{\predec}{prec}

\DeclareMathOperator{\ro}{ro}
\DeclareMathOperator{\cl}{cl}

\DeclareMathOperator{\dom}{dom}

\DeclareMathOperator{\cf}{cf}
\DeclareMathOperator{\NS}{NS}
\DeclareMathOperator{\seq}{seq}
\newcommand{\n}{\underaccent{\tilde}}
\newcommand{\name}[1]{\n{#1}}
\newcommand{\esm}{\prec}
\newcommand{\forc}{\Vdash}
\newcommand{\pow}{\mathfrak{P}}

\newcommand{\propa}{\mathfrak{a}}
\newcommand{\propb}{\mathfrak{b}}

\newcommand{\al}[1]{\aleph_{#1}}
\newcommand{\om}[1]{\omega_{#1}}
\newcommand{\cac}{\textrm{c\&c}}
\newcommand{\cacset}{\textrm{c\&c}^\textrm{set}}
\newcommand{\Gmin}{\textrm{c\&c}^\textrm{min}}
\newcommand{\Gminne}{\textrm{c\&c}^\textrm{min}_\textrm{ne}}

\newcommand{\BMIe}{\textrm{BM}}
\newcommand{\BMze}{\textrm{Id}}
\newcommand{\BMIne}{\textrm{BM}_{\textrm{ne}}}
\newcommand{\BMzne}{\textrm{Id}_{\textrm{ne}}}
\newcommand{\PDIe}{\textrm{PD}_{\textrm{e}}}
\newcommand{\PDze}{\textrm{PD}^{\emptyset}_{\textrm{e}}}
\newcommand{\PDIne}{\textrm{PD}}
\newcommand{\PDzne}{\textrm{PD}^{\emptyset}}

\begin{document}

\date{2006-09-14}
\title{More on the pressing down game.}

\author[Jakob Kellner]{Jakob Kellner$^\ast$}
\address{Kurt G\"odel Research Center for Mathematical Logic\\
  Universit\"at Wien\\
  W\"ahringer Stra\ss e 25\\
  1090 Wien, Austria}
\email{kellner@fsmat.at}
\urladdr{http://www.logic.univie.ac.at/$\sim$kellner}
\thanks{$^\ast$ supported by European Union FP7 grant PERG02-GA-2207-224747 and the Austrian FWF project P21651-N13.}
\author[Saharon Shelah]{Saharon Shelah$^\dag$}
\address{Einstein Institute of Mathematics\\
  Edmond J. Safra Campus, Givat Ram\\
  The Hebrew University of Jerusalem\\
  Jerusalem, 91904, Israel\\
  and
  Department of Mathematics\\
  Rutgers University\\
  New Brunswick, NJ 08854, USA}
\email{shelah@math.huji.ac.il}
\urladdr{http://shelah.logic.at/}
\thanks{$^\dag$ supported by
  the United States-Israel Binational Science Foundation (Grant no. 2002323),
  publication 939.}

\begin{abstract}
  We investigate the pressing down game and its relation to the Banach Mazur
  game. In particular we show: Consistently, there is
  a nowhere precipitous normal ideal $I$ on $\aleph_2$ such that player nonempty
  wins the pressing down game of length $\aleph_1$ on $I$ even if player empty
  starts.
\end{abstract}

\maketitle

We investigate the pressing down game and its relation to the Banach Mazur
game. Definitions (and some well known or obvious properties) are given in
Section~\ref{sec:def}.  The results are summarized in
Section~\ref{sec:results}. This paper continues (and simplifies,
see~\ref{lem:burp2}) the investigation of Pauna and the authors in
\cite{MR2371208}.

We thank the referee for kindly pointing out an error and an embarrasingly large
number of typos.

After the submission of this paper it came to our attention that
Gitik~\cite{MR2414461} already proved Fact~\ref{thm:oiuqweh} of this paper,
moreover he just requires a measurable cardinal (we use a supercompact).
Nevertheless we include our proof in this paper, maybe the construction could
be of intrerest in other situations.

\section{Definitions}\label{sec:def}

We use the following notation:
\begin{itemize}
  \item For forcing conditions $q\leq p$, the smaller condition $q$
    is the stronger one. We stick to Goldstern's alphabetic
    convention~\cite[1.2]{MR1601976}: Whenever two conditions are comparable the
    notation is chosen so that the variable used for the stronger condition
    comes ``lexicographically'' later.
  \item
    $E^\kappa_\lambda=\{\alpha\in\kappa:\, \cf(\alpha)=\lambda\}$.
  \item
    $\NS_\kappa$ is the nonstationary ideal on $\kappa$.
  \item The dual of an ideal $I$ is the filter $\{A\subseteq \kappa:\,
    \kappa\setminus A\in I\}$ and vice versa.
  \item For an ideal $I$ on $\kappa$ 
    and a positive set $A$ (i.e., $A\notin I$),
    we set
    $I\restriction A$ to be the ideal generated by
    $I\cup \{\kappa\setminus A\}$.
\end{itemize}

We always assume that $\kappa$ is a regular uncountable cardinal and that $I$
is a ${<}\kappa$-complete ideal on $\kappa$.  Unless noted otherwise,
we will also assume that $I$ is normal.

We now recall the definitions of several games of length $\omega$, played by
the players empty and nonempty.  We abbreviate ``having a winning strategy for
$G$'' with ``winning $G$'' (as opposed to: ``winning a specific run of $G$'').

First we define four variants of the pressing down game (this game
has been used, e.g., in~\cite{MR1191613}).
\begin{Def}
  \begin{itemize}
    \item $\PDIne(I)$ is played as follows:
      Set $S_{-1}=\kappa$.  At stage $n$, empty chooses a regressive function
      $f_n:\kappa\to\kappa$, and nonempty
      chooses $S_n$, an $f_n$-homogeneous $I$-positive subset of $S_{n-1}$.
      Empty wins the run of the game if $\bigcap_{n\in\omega} S_n\in I$.
    \item $\PDzne(I)$ is played like $\PDIne(I)$, but empty wins
      the run if $\bigcap_{n\in\omega} S_n=\emptyset$.
    \item $\PDIe(I)$ is played like $\PDIne(I)$, but empty can first choose
      $S_{-1}$ to be an arbitrary $I$-positive set.
    \item $\PDze$ is defined analogously.
  \end{itemize}
\end{Def}

So we have four variants of the pressing down game, depending on two
parameters: whether the winning condition for player nonempty is
``${\neq}\emptyset$'' or ``${\notin}I$'', and whether empty has
the first move or not.

We now analogously define four variants of the Banach Mazur game:
\begin{Def}
  \begin{itemize}
    \item $\BMIe(I)$ is played as follows: Set
      $S_{-1}=\kappa$.  At stage $n$, empty chooses
      an $I$-positive subset $X$ of $S_{n-1}$, and nonempty chooses
      an $I$-positive subset $S_n$ of $X$.  Empty wins the run if
      $\bigcap_{n\in\omega} S_n\in I$.
    \item The ideal game
      $\BMze(I)$ is played just like $\BMIe(I)$, but empty
      wins the run if $\bigcap_{n\in\omega} S_n=\emptyset$.
    \item $\BMIne(I)$ is played just like
      $\BMIe(I)$, but nonempty has the first move.
    \item $\BMzne(I)$ is defined analogously.
  \end{itemize}
\end{Def}

More generally, we can define the Banach Mazur game $\BMIe(B)$ on a Boolean
algebra $B$: The players choose decreasing (nonzero) elements $a_n\in B$,
nonempty wins if there is some (nonzero) $b\in B$ smaller than all $a_n$.  Then
$\BMIe(I)$ is equivalent to the corresponding game $\BMIe(B_I)$ on the Boolean
algebra $B_I=\pow(\kappa)/I$ (since $I$ is $\sigma$-complete), the same holds for
$\BMIne(I)$ and $\BMIne(B_I)$; we could equivalently use the completion
$\ro(B_I)$ instead of $B_I$.
Also the ${\notin}I$ versions of the pressing down game can be played modulo
null sets, i.e., on the Boolean algebra $B_I$, in the obvious way.
For the ${\neq}\emptyset$ versions of the games, the version played
on $B_I$ does not make sense.

In the ${\notin}I$ version, the pressing down and Banach Mazur games have
natural generalizations to other lengths $\delta$: At a limit stages
$\gamma$, we use $\bigcap_{\alpha<\gamma} S_\alpha$ instead of
$S_{\gamma-1}$, and empty wins a run iff this set is in
$I$ for any $\gamma<\delta$.  (I.e., nonempty wins a run iff the run has length
$\delta$. So in this setting, the games defined above are the ones of length
$\omega+1$.) For the ${\neq}\emptyset$ versions of the games, lengths other than
$\omega+1$ seem less natural.

We are interested in the existence of winning strategies:
\begin{Def}
  \begin{itemize}
    \item We write $\propb(G)$ for ``nonempty wins $G$'' and
      $\propa(G)$ for ``empty does not win $G$''.
    \item The games $G$ and $H$ are equivalent, if
      $\propb(G)\leftrightarrow \propb(H)$ and
      $\propa(G)\leftrightarrow \propa(H)$.
    \item $G$ is stronger than $H$, if
      $\propb(G)\rightarrow \propb(H)$ and
      $\propa(G)\rightarrow \propa(H)$.
  \end{itemize}
\end{Def}

We trivially get the following implications,
see Figure~\ref{fig:eins}:
\begin{Facts}
  \begin{itemize}
    \item $\propb(G)\rightarrow\propa(G)$ for all games.
    \item The Banach-Mazur game is stronger than
          the according pressing down game. E.g.,
          $\BMIne(I)$ is stronger than $\PDIne(I)$ etc.
    \item The ${\notin}I$ version is stronger than the
          ${\neq}\emptyset$ one. E.g.,
          $\BMIe(I)$ is stronger than $\BMze(I)$ etc.
    \item The version with empty choosing first is stronger. E.g.,
          $\BMIe(I)$ is stronger than $\BMIne(I)$ etc.
  \end{itemize}
\end{Facts}

We now list some well known (or otherwise obvious) facts 
about BM and precipitous ideals\footnote{these facts do not require
  that $I$ is normal}, see~\cite{MR0505504,MR0485391,MR560220}:
\begin{Facts}
  \begin{itemize}
    \item $\propa(\BMze(I))$ is equivalent to ``$I$ is precipitous''.
    \item $\propa(\BMzne(I))$ is sometimes called ``$I$ is somewhere
      precipitous'', and its failure ``$I$ is nowhere precipitous''.
    \item A precipitous ideal on $\kappa$
      implies that $\kappa$ is measurable in an inner model.
    \item $\propb(\BMIe(\NS_{\al2}\restriction E^{\al2}_{\al1}))$ is
      equiconsistent to a measurable.
    \item
      ``$\NS_{\al1}$ is precipitous'' is also equiconsistent to a measurable.
    \item $\propb(\BMze(I))$ implies $\kappa>2^{\al0}$ and $E^\kappa_{\al0}\in I$.
    \item $\propa(\BMIe(I))$ implies $E^\kappa_{\al0}\in I$, and in
      particular $\kappa>\al1$.
  \end{itemize}
\end{Facts}

\begin{figure}[tb]
  \centerline{\xymatrix@!C=0.6cm{
    &
    \propb(\BMIe)&
    &
  \\
    \propb(\BMze)&
    \propa(\BMIe)&
    \propb(\PDIe)&
  \\
    \propa(\BMze)&
    \propb(\PDze)&
    \propa(\PDIe)&
  \\
    &
    \propa(\PDze)&
    &
  \ar"1,2";"2,1"
  \ar"1,2";"2,2"
  \ar"1,2";"2,3"
  \ar"2,1";"3,1"
  \ar"2,1";"3,2"
  \ar"2,3";"3,2"
  \ar"2,3";"3,3"
  \ar"3,2";"4,2"
  \ar"2,2";"3,1"
  \ar"2,2";"3,3"
  \ar"3,1";"4,2"
  \ar"3,3";"4,2"
  }}
  \caption{The trivial implications (for empty moving first)}
  \label{fig:eins}
\end{figure}
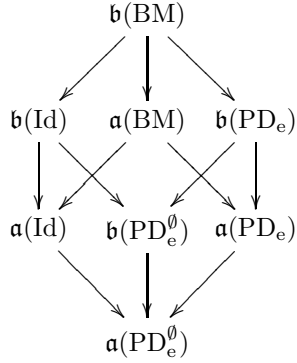

Some obvious facts about PD (for normal ideals $I,J$):
\begin{Facts}
  \begin{itemize}
    \item In the pressing down games,
      we can assume without loss of generality that nonempty
      chooses at stage $n$ a set of the form $S_n=f_n^{-1}(\alpha_n)\cap
      S_{n-1}$ for
      some $\alpha_n$.\footnote{This of course means: PD is equivalent to the
      game where nonempty is restricted to moves of this form.}
    \item
      $\PDIne$ is monotone in the following sense: if $J\supseteq I$, then
      $\PDIne(J)$ is stronger than $\PDIne(I)$.
      The same holds for $\PDzne$, but not for
      $\PDIe$ nor $\PDze$ nor for any of
      the Banach Mazur games.
    \item In particular, $\PDIne(I)$ is stronger than $\PDIne(\NS_\kappa)$ 
      for all normal $I$.
    \item Just as in the case of $\BMIe$, $\propb(\PDzne)$ cannot hold
      for $\kappa=\al1$ (cf.~\ref{lem:bPDgtc}). 
    \item Other than
      in the case of $\BMze$, the property $\propa(\PDIe)$ has no
      consistency strength (cf.~\ref{lem:jhw3t23}). 
  \end{itemize}
\end{Facts}

What is the effect of empty moving first?
\begin{Facts}\label{facts:density}
  \begin{itemize}
    \item For the Banach-Mazur games, the distinction whether empty has the
      first
      move or nonempty is a simple density effect: For example, nonempty wins
      $\BMIne(I)$ iff there is some $S\in I^+$ such that nonempty  wins
      $\BMIe(I\restriction S)$; similarly simple equivalences hold for empty
      winning; for characterizing
      $\BMIe$ in terms of $\BMIne$; and for the ${\neq}\emptyset$ version.
    \item We will see in Lemma~\ref{lem:akjwrt} that this is not the case for the pressing down
      game.
  \end{itemize}
\end{Facts}

The ${\notin}I$ versions of BM and PD are in fact instances of the cut and
choose game introduced by Jech \cite{MR739910} (and its ancestor, the Ulam
game):

\begin{Def}
  The {\bf cut and choose} game $\cac(B,\lambda)$
  on a Boolean algebra $B$ is played as follows:
  First empty chooses  a nonzero element $a_0$ of $B$.
  At stage $n$, empty chooses
  a maximal antichain $A_n$ below $a_n$ of size at most $\lambda$,
  and nonempty chooses an element $a_{n+1}$ from $A_n$.
  Nonempty wins the run if there is some nonzero $b$ below
  all $a_n$.\\
  $\cac(B,\infty)$ is played without restriction on the size of
  the antichains.
\end{Def}

Let $\ro(B)$ denote the completion of the Boolean algebra $B$, and set $B_I=\pow(\kappa)/I$.
The following can be found, e.g.,
in~\cite{MR739911,MR716633,MR846604,MR1366514}:
\begin{Facts}\label{facts:fa732}
  \begin{itemize}
    \item 
        $\cac(B,\infty)$ is equivalent to $\cac(\ro(B),\infty)$.\footnote{%
        But $\cac(B,\lambda)$ will generally not be equivalent to
        $\cac(\ro(B),\lambda)$.}
    \item 
      $\cac(B,\infty)$ is equivalent to the Banach Mazur game on $B$. 
    \item In particular,
      $\cac(B_I,\infty)$ is equivalent to $\BMIe(I)$.
    \item
      $\cac(B_I,\kappa)$ is equivalent to $\PDIe(I)$, cf.~\ref{lem:gurktu}.
      \\
      (However $\cac(\ro(B_I),\kappa)$ might be a stronger game.)
  \end{itemize}
\end{Facts}

It is less clear how the ${\neq}\emptyset$-versions of BM and PD relate
to possible set-versions of the cut-and-choose game. On natural
candidate is a ``set-partition'' game:
\begin{Def}\label{def:Gmin}
  \begin{itemize}
    \item $\Gmin(I,\lambda)$ is played as follows: First, empty chooses 
      some  positive $S_{-1}$. At stage $n$, empty partitions the set
      $S_{n-1}$ into at most $\lambda$ many (arbitrary) pieces, and nonempty
      chooses an $I$-positive%
\footnote{Note that we allow empty to include $I$-null pieces into the
partition, but we require nonempty to choose a positive piece; otherwise
nonempty always wins by picking right from the start an element $\alpha$ and
then always picking the piece containing $\alpha$.}
      piece $S_n$.
      Empty wins the run iff $\bigcap_{n\in\omega} S_n=\emptyset$.
    \item In $\Gmin(I,{<}\kappa)$
      empty cuts into less than $\kappa$ many (arbitrary) pieces.
    \item $\Gminne$ is defined as usual, i.e., $S_{-1}=\kappa$
  \end{itemize}
\end{Def}

The following is straigthforward:
\begin{Fact}
  $\PDze(I)$ is stronger than $\Gmin(I,{<}\kappa)$,
  and $\Gmin(I,{<}\kappa)$ is stronger than $\Gmin(I,2)$.
\end{Fact}

\begin{Rem}
  Another variant: Empty has to partition into {\em positive} pieces
  (of size at most $\lambda$),
  and wins a run iff the intersection is in $I$. Let us call this
  game $\cacset(I,\lambda)$ (we will not need it in the rest of the paper).
  It is not entirely clear how this game relates to the previous ones:
  \\
  Obviously there can be at most $\kappa$ many pieces, so
  $\cacset(I,\infty)=\cacset(I,\kappa)$. 
  \\
  For $\lambda<\kappa$ it is easy
  to see that $\cacset(I,\lambda)$ is equivalent to $\cac(B_I,\lambda)$.
  \\
  Also, it is clear that $\cacset(I,\kappa)$ is stronger than
  $\cac(B_I,\kappa)$, which is equivalent to $\PDIe(I)$.
  \\
  The relation of $\cacset(I,\kappa)$ and $\BMIe(I)$ is less clear.
  Of course, 
  if $I$ is $\kappa^+$-saturated, then $\cacset(I,\kappa)$,
  $\cac(B_I,\kappa)$, $\PDIe(I)$ and $\BMIe(I)$ are all equivalent,
  cf.~\ref{lem:gurktu} and~\ref{cor:satpdbm}.
\end{Rem}

Winning strategies for games on a Boolean algebra $B$ have close connections to
the properties of $B$ as Boolean algebra and as forcing notions, again
see~\cite{MR739911,MR716633,MR846604,MR1366514}:
\begin{Facts}
  \begin{itemize}
    \item $B$ having a $\sigma$-closed positive subset
        implies $\propb(\BMIe(B))$.
    \item $\propb(\BMIe(B))$ is also denoted by
        ``$B$ is strategically $\sigma$-closed'' and
        implies that $B$ is proper.
    \item $\propa(\BMIe(B))$ is equivalent
        to ``$B$ is $\sigma$-distributive''.
  \end{itemize}
\end{Facts}
It is not surprising that we will get stronger connections if we assume that
the $B$ has the form $B_I=\pow(\kappa)/I$ for a normal ideal $I$.  
We will mention only one example:
\begin{Fact}
  If $B_I$ is proper and $\kappa>2^{\al0}$, then $\propa(\BMIe(I))$ holds.
\end{Fact}
For a proof, see~\ref{lem:propaBM}.

\section{The results}\label{sec:results}

Some of the facts for precipitous ideals can be shown (with similar proofs)
for PD, but there are of course strong differences as well:
\begin{Lem}\label{lem:jhw3t23}
  \begin{enumerate}
    \item $\propb(\Gminne(I,2))$ implies that $\kappa$ is measurable in an inner
      model. 
    \item So in particular, $\propb(\PDzne(I))$ implies that as well.
    \item However, $\propa(\PDIe(I))$ has no consistency strength.
      In particular, for $\kappa=\al2$,
      $\propa(\PDIe(I))$ is implied by $CH$ for every $I$ concentrated
      on $E^{\al2}_{\al1}$.
    \item $\propb(\PDzne(I))$ implies $\kappa>2^{\al0}$ and
      that $I$ is not concentrated on $E^\kappa_{\al0}$.
  \end{enumerate}
\end{Lem}

The proofs can be found in \ref{lem:bPDgne}, \ref{lem:bPDgtc}
and \ref{lem:uzligutzli}.

In this paper, we are not interested in the property ``empty does not win the
pressing down game'', since it has no consistency strength.
Also, the effect of who moves first in Banach Mazur games is trivial.
The remaining properties are pictured in Figure~\ref{fig:zwei}.
\begin{figure}[tb]
\centerline{\xymatrix@!R=0.1cm{
    \propb(\BMIe)&
    &
    \propb(\PDIe)&
    \propb(\PDIne)&
  \\
    &
    \propa(\BMIe)&
    &
    &
  \\
    \propb(\BMze)&
    &
    \propb(\PDze)&
    \propb(\PDzne)&
  \\
    &
    \propa(\BMze)\equiv\textrm{precip.}&
    &
    \infty\textrm{-semi precipitous}&
  \ar"1,1";"3,1"
  \ar"1,1";"2,2"
  \ar"1,1";"1,3"
  \ar"1,3";"3,3"
  \ar"1,3";"1,4"
  \ar"1,4";"3,4"
  \ar"2,2";"4,2"
  \ar"3,1";"3,3"
  \ar"3,1";"4,2"
  \ar"3,3";"3,4"
  \ar"3,4";"4,4"
  \ar"4,2";"4,4"
}}
\caption{Some properties stronger than $\infty$-semi precipitous}
\label{fig:zwei}
\end{figure}
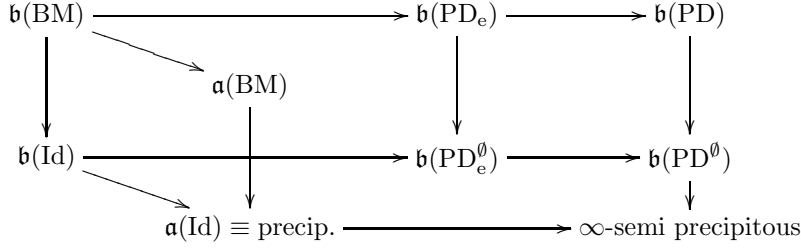
All these properties are equiconsistent to a measurable (e.g., for
$I=\NS_{\al2}\restriction E^{\al2}_{\al1}$).  In fact, they
imply that $I$ is $\infty$-semi precipitous, see
Definition~\ref{def:inftysemi}, which in turn implies that $\kappa$
is measurable in an inner model.
We claim that none of the implications can be reversed. In this paper, we will
prove some strong instances of this claim by assuming larger
cardinals: We show
\begin{itemize}
  \item $\propb(\PDIe)$ does not imply precipitous, and 
  \item $\propa(\BMIe)$ does not imply $\propb(\PDzne)$.
\end{itemize}
We also claim that (consistently relative to a measurable)
\begin{itemize}
  \item $\propb(\BMze)$ does not imply $\propb(\PDIne)$,
\end{itemize}
but we do not give a proof here.  With these claims (for which we assume
cardinals larger than a measurable) it is then easy to check that no
implication of Figure~\ref{fig:zwei} can be reversed.

In \cite{MR2371208},  Pauna and the authors showed that, assuming the
consistency of a measurable, $\propb(\PDIne(I))$ does not imply $\propb(\BMIne(I))$ for
$I=\NS_{\al2}\restriction E^{\al2}_{\al1}$. In fact, a slightly stronger
statement holds (with a simpler proof):
\begin{Lem}\label{lem:burp2}
  It is equiconsistent with a measurable 
  that $\propb(\PDIne(I))$ holds (even for length $\om1$)
  but $\propa(\BMzne(I)))$ fails
  for $I=\NS_{\al2}\restriction E^{\al2}_{\al1}$.
\end{Lem}
(For a proof, see~\ref{cor:gurk422}.)
Note that ``$\propa(\BMzne(I)))$ fails''
just means that $I$ is nowhere precipitous.

Of course, precipitous cannot generally 
imply a winning strategy for nonempty in any
game, since precipitousness is consistent with $\kappa\leq 2^{\al0}$.
However, we can get counterexamples for $\kappa>2^{\al0}$ as well:
Just adding Cohens destroys any winning strategy for nonempty
(for any ideal on $\al2$), but preserves
precipitous. So we get (see~\ref{lem:fueel} and~\ref{cor:funfz}):
\begin{Lem}\label{lem:burp3}
  It is equiconsistent with a measurable that
  CH holds, $\NS_{\al2}\restriction E^{\al2}_{\al1}$ is precipitous 
  but  $\propb(\PDzne(J))$ fails for any normal ideals $J$ on $\al2$.
\end{Lem}

To see that not even $\propa(\BMIe(I))$ implies any winning strategy
for nonempty, we assume CH and a $\al3$-saturated ideal $I$ on
$\al2$ concentrated on $E^{\al2}_{\al1}$.
Saturation is preserved by small forcings, in particular by adding some Cohens,
and saturation (together with CH) implies $\propa(\BMIe(I))$.
So we get:
\begin{Lem}\label{lem:burp4}
  The following is consistent with CH plus an $\al3$-saturated ideal  
  on $\al2$:
  CH holds, $\propa(\BMIe(I))$ holds for some $I$ on $\al2$,  but 
  $\propb(\PDzne(J))$ fails for any normal $J$ on $\al2$.
\end{Lem}
See~\ref{cor:funfe}.
(It seems very likely that saturation is not needed for this, but the
construction might get considerably more complicated without it.)

As mentioned in Lemma~\ref{lem:burp2} it is possible that $\propb(\PDIne(I))$
holds for a nowhere precipitous ideal, i.e., for an ideal such that
$\propa(\BMzne(I))$ fails. With a bit more work, we even get
$\propb(\PDIe(I))$:
\begin{Thm}
  It is equiconsistent with a measurable that for  $\kappa=\al2$
  there is a nowhere precipitous $I$  such that
  $\propb(\PDIe(I))$ holds (even for length $\om1$).
\end{Thm}
(See Fact~\ref{thm:oiuqweh}.)

Note that (as opposed to~\ref{lem:burp3},~\ref{lem:burp4}) we just make a
specific ideal non-precipitous, and we do not destroy all precipitous ideals.
It seems very hard (and maybe impossible) to do better:
It is not known how to kill all precipitous ideals%
\footnote{Since we are only interested in normal ideals, it would be enough to
  kill all normal precipitous ideals. This doesn't help much, though; it is not
  known whether the existence of a precipitous ideal does imply the existence of
  a normal precipitous one.  Recently Gitik
  \cite{MR2414461,gitik,onalmostprec} proved some interesting results in
  this direction.%
}
on, e.g., $\al1$ with ``reasonable'' forcings.%
\footnote{More specifically,
  it is not known whether large cardinals imply a precipitous
  ideal on $\al1$, although Woodins are not enough, cf.~\cite{MR2029599}.%
}
And it might be even harder to do so while additionally preserving
$\propb(\PDIe(I))$ for some ideals: By recent results by Gitik \cite{gitik}
(and later Ferber and Gitik \cite{onalmostprec}) a $\infty$-semi precipitous
ideal does imply a normal precipitous ideal under in the absence of larger
cardinals and under some cardinal arithmetic assumptions.

\subsection{Moving first}

Let us now investigate the effect of whether empty moves first.

If we compare $G_\text{e}$ and $H_\text{ne}$ for any games $G$ and $H$,
then these variants will be different for trivial reasons: For
example, $\propb(\BMIne(I))$ does not imply
$\propb(\PDze(I))$: Let $U$ be a normal ultrafilter on $\kappa$,
Levy-collapse $\kappa$ to $\al2$, and let $I_1$ be the ideal generated by the
dual of $U$ (which is concentrated on $E^{\al2}_{\al1}$).  Then nonempty wins
$\BMIe(I_1)$ and therefore $\BMIne(I)$ as well for 
$I=I_1+\NS_{\al2}\restriction E^{\al2}_{\al0})$ as well.  But nonempty
can never win $\PDze(I)$, since nonempty cannot win
$\PDzne(\NS_{\al2}\restriction E^{\al2}_{\al0})$.
The same holds for $I=\NS_{\al2}$ (just use the model of
$\propb(\BMIe(\NS_{\al2}\restriction E^{\al2}_{\al1}))$.)

So the games are very different (for trivial reasons) when we change who
has the first move. However, for the Banach Mazur game, the
effect of who moves first is
a simple density effect,\footnote{In games of
  length bigger than $\omega+1$ however it does make a
  substantial difference who moves first at limits.%
}
as we have mentioned in \ref{facts:density}. For example, $\propb(\BMIne(I))$
holds iff $\propb(\BMIe(I\restriction S))$ holds for some positive $S$.

This is not the case for the pressing down games. 
Of course we still get:
\begin{itemize}
  \item $\propb(\PDIe(I))$ holds iff $\propb(\PDIne(I\restriction S))$ holds
    for all $S\in I^+$.
  \item The same holds for $\PDze$.
\end{itemize}

But unlike the Banach Mazur case, we can have the following:
\begin{Lem}\label{lem:akjwrt}
  It is equiconsistent with a measurable that
  $\propb(\PDIne(I))$ holds but 
  $\propb(\PDIe(I\restriction S))$ fails 
  for all positive $S$,
  e.g., for $I=\NS_{\al2}$.
\end{Lem}
(See~\ref{cor:gurk422}.)
So in other words, $\propb(\PDIne(I))$ can hold but for all
positive $S$ there is a positive $S'\subseteq S$ such that
$\propb(\PDIne(I\restriction S'))$ fails.

\section{Empty not winning}
\begin{Lem}\label{lem:ar2}
  \begin{itemize}
    \item CH implies $\propa(\PDIne(I))$ for every $I$ on $\al2$ that is not
      concentrated on $E^{\al2}_{\al0}$.
    \item
      More generally, if $\lambda^{\al0}<\kappa$ for all $\lambda<\kappa$, then
      empty wins $\PDIne(I)$ iff $E^\kappa_{{>}\omega}\in I$.
    \item So if $I$ is concentrated on $E^\kappa_{>\omega}$ (and
      the same cardinal assumptions hold) then $\propa(\PDIe(I))$ holds. 
  \end{itemize}
\end{Lem}

\begin{proof}
  Assume that $I$ is concentrated on $E^\kappa_{\al0}$.
  Just as in \cite{MR0485391}, it is easy to see that empty wins $\PDIne(I)$:
  For every $\alpha\in E^\kappa_{\al0}$, let
  $(\seq(\alpha,n))_{n\in\omega}$ be a cofinal sequence in $\alpha$.
  Let $F_n$ map $\alpha$ to $\seq(\alpha,n)$. If empty plays
  $F_n$ at stage $n$, then the intersection can contain at most one
  element. 

  So assume towards a contradiction that
  $E^\kappa_{{>}\omega}\notin I$ and that empty has a winning
  strategy for $\PDIne(I)$.
  The strategy assigns sets $X_t$ and regressive function $f_t$
  to nodes $t$ in the tree $T=\kappa^{<\omega}$ in the following way:

  For $t=\langle\rangle$, set $X_{\langle\rangle}=\kappa$ and
  let $f_{{\langle\rangle}}$ be empty's first move.
  For $\alpha\in \kappa$, set $X_{(\alpha)}=f_t^{-1}(\alpha)$.
  Note that $\alpha$ is a valid response for nonempty iff
  $X_{(\alpha)}$ is positive.
  Generally, fix $t\in T$. We can assume by induction that
  one of the following cases hold:
  \begin{itemize}
    \item $t$ corresponds to a partial run $r_t$ with (positive) partial
      result $X_t$; then we set $f_t$ to be empty's response to
      $r_t$.
    \item $X_t\in I$; then we set $f_t\equiv 0$.
  \end{itemize}
  In both cases we set
   $X_{t^\frown \alpha}=X_t\cap f_t^{-1}(\alpha)$.

  Let $b$ be a branch of $T$ (i.e.,  $b\in\kappa^\omega$).
  We set
  $ X^b=\bigcap_{n\in\omega} X_{b\restriction n}$.

  Assume that $b$ corresponds to a run of the game; this is the case iff
  $X_{b\restriction n}$ is $I$-positive for all $n$.
  Then $X^b\in I$, since empty uses the winning strategy.
  If $b$ does not correspond to a run, then $X^b\in I$ as well. So
  \begin{equation}\label{eq:allnull}
    X^b\in I\text{ for all branches }b
  \end{equation}
  $X^b$ and $X^c$ are disjoint for different branches $b,c$;
  and for all $\gamma\in\kappa$ there is exactly one branch $b_\gamma$
  such that $\gamma\in X^{b_\gamma}$. We assume $\gamma\neq 0$ from now on.
  By definition, for all $n$
  \[
    f_{b_\gamma\restriction n}(\gamma)=b_\gamma(n)
  \]
  Since $f_{b_\gamma\restriction n}$ is regressive,
  $b_\gamma(n)<\gamma$ for all $n\in\omega$. In other words,
  $b_\gamma\in\gamma^\omega$.

  Fix an injective function
    $\phi:\kappa^{\omega}\to \kappa$.
  Since $\gamma^{\al0}<\kappa$ for $\gamma<\kappa$, we can find
  a club $C$ such that
  \[
    \phi''\gamma^\omega\subseteq \gamma\text{ for all }
    \gamma\in C\cap E^\kappa_{{>}\omega}.
  \]
  This defines a regressive function
  $g:C\cap E^\kappa_{{>}\omega}\to\kappa$ by
  $g(\gamma)=\phi(b_\gamma)$. Since $I$ is normal and
  does not contain $E^\kappa_{{>}\omega}$, there is a positive
  set $S$ and a $\zeta\in\kappa$ (or equivalently a branch $b$ of $T$)
  such that $g(\gamma)=\zeta$, i.e., $b_\gamma=b$ for all $\gamma\in S$.
  This implies that $S\subseteq X^b$ is positive, a contradiction to~\eqref{eq:allnull}.
\end{proof}

\begin{Lem}\label{lem:gurktu}
  If $I$ is normal, then
  $\PDIe$ is equivalent to $\cac(B_I,\kappa)$.
\end{Lem}

\begin{proof}
  A regressive function defines a maximal antichain
  in $B_I$ of size at most $\kappa$.
  On the other hand, let $A$ be a maximal antichain of size $\lambda\leq
  \kappa$.
  We can choose pairwise disjoint representatives $(S_i)_{i\in\lambda}$
  for the elements of $A$, and define
  \[
     f(\alpha)=
     \begin{cases}
       1+i & \text{if } \alpha\in S_i\text{ and }1+i<\alpha,\\
       0   & \text{otherwise.}
     \end{cases}
  \]
  $f^{-1}(0)\in I$. (Otherwise there is an $S_i$ in $A$ such
  that $T=S_i\cap f^{-1}(0)\in I^+$, pick $\alpha\in T\setminus (1+i+1)$,
  contradiction.)
  So the partition $A$ is equivalent to the regressive function $f$.
\end{proof}

Together with \ref{facts:fa732} we get:
\begin{Cor}\label{cor:satpdbm}
  If $I$ is normal and $\kappa^+$-saturated, then 
  $\BMIe(I)$ and  $\PDIe(I)$ are equivalent. The same holds for
  $\BMIne(I)$ and $\PDIne(I)$.
\end{Cor}

If $I$ is $\kappa^+$-saturated, then it is precipitous, i.e.,
$\propa(\BMze(I))$ holds \cite[22.22]{MR1940513}.  However, $I$ can
be concentrated
on $E^\kappa_{\al0}$ (for example, $\kappa$ could be $\al1$), which negates
$\propa(\BMIe(I))$.  However, with Lemma~\ref{lem:ar2} we get:
\begin{Cor}\label{cor:wrqw}
  If $I$ is $\kappa^+$-saturated, $\lambda^{\al0}<\kappa$ for
  all $\lambda<\kappa$, and $I$ is normal and 
  concentrated on $E^\kappa_{{>}\al0}$,
  then $\propa(\BMIe(I))$ holds.
\end{Cor}

In the rest of the section, we show that properness implies $\propa(\BMIe(I))$. 
This is not needed for the rest of the paper.

For any Boolean algebra $B$, $\propb(\BMIe(B))$ implies that $B$ is proper (as
a forcing notion), cf.\ e.g.~\cite[Thm.\ 7]{MR739910}.  For Boolean algebras of
the form $B_I=\pow(\kappa)/I$ we also get:
\begin{Lem}\label{lem:propaBM}
  Assume $\kappa>2^{\al0}$. If $B_I$ is proper then
  $\propa(\BMIe(I))$ holds. 
\end{Lem}
Normality of $I$ is not needed, just ${<}\kappa$-completeness.

\begin{proof}
  Assume
  towards a contradiction that $\tau$ is a winning strategy for empty.
  Let $p_0\in I^+$ be empty's first move according to $\tau$.
  Pick $N\esm H(\chi)$ countable containing $I$ and $\tau$ (and therefore $p_0$),
  and let $q\leq p_0$ be
  $N$-generic. In other words,
  if $\mathcal D\in N$ is a predense subset of $I^+$,
  then $q$ is (mod $I$) a subset of $\bigcup(\mathcal  D\cap N)$.
  Therefore
  \[
    \mathcal X=q\cap \bigcap\{\bigcup (\mathcal D\cap N):\,
    \mathcal D\subseteq I^+\text{ is predense and }\mathcal D\in N\},
  \]
  is positive. We set
  \[
    \mathcal Y=\bigcup \{\bigcap_{n\in\omega}A_n:\,
      (\forall n\in\omega)\,A_n\in N\cap I^+,\, \bigcap_{n\in\omega}A_n\in I\}
  \]
  $\mathcal Y\in I$, since $|[N]^{\al0}|<\kappa$. So we can pick some
  \[
    \delta^*\in \mathcal X\setminus \mathcal Y.
  \]
  We now construct a run of the game such that every
  initial segment is in $N$.
  Assume that we already know the initial segment of the first $n-1$ stages,
  and that this segment is in $N$. Then empty's move $A_n$ given by
  $\tau$ is in $N$ as well. We further assume that $\delta^*\in A_n$.
  (This is true for $n=0$, since $\delta^*\in q\leq p_0$.)
  For any $I$-positive $B\subseteq A_n$ let empty's response be
  $f(B)$. The set
  \[
    D=\{\kappa\setminus A_n\}\cup
             \{f(B): B\subseteq A_n\text{ positive}\}
  \]
  is dense in $I^+$ and is in $N$. Since $\delta^*\in\mathcal X$,
  $\delta^*\in \bigcup (D\cap N)$,
  i.e.\ there is some $B\in N$ such that
  $\delta^*\in f(B)$.
  Let $B$ be nonempty's move.

  So $\delta^*$ will be in the intersection $Z=\bigcap_{n\in\omega} A_n$, and
  since empty wins the run, $Z\in I$.
  Since each $A_n$ is in $N$, we get $Z\subseteq \mathcal Y$.
  This contradicts $\delta^*\in Z$.
\end{proof}

\section{$\infty$-semi precipitous ideals}
\begin{Def}\label{def:inftysemi}
  A $\kappa$-complete ideal  $I$ on
  $\kappa$ is called
  (normally) $\infty$-semi precipitous, if there is some partial order
  $P$ which
  forces  that there is a (normal) wellfounded, nonprincipal,
  $\kappa$-complete $V$-ultrafilter containing the dual
  of $I$.
\end{Def}
Donder, Levinski \cite{MR1026565} introduced the notion
of $\lambda$-semi precipitous, and Ferber and Gitik  \cite{onalmostprec}
extended the notation to $\infty$-semi precipitous. Another name,
``weakly precipitous'', is used for this notion in
\cite{MR1240627}. However, Jech uses the term ``weakly precipitous'' for
another concept, cf.~\cite{MR739911, MR1026565}.

We will see in Lemma~\ref{lem:bPDgne} that $\propb(\PDzne(I))$ implies
that $I$ is normally $\infty$-semi precipitous. This will establish
the consistency strength of $\propb(\PDzne(I))$:

\begin{Lem}\label{wurst}
  If there is an $\infty$-semi precipitous ideal on $\kappa$,
  then $\kappa$ is measurable in an inner model.
\end{Lem}

This is of course no surprise: the proof is a simple
generalization of the proof~\cite[Theorem 2]{MR560220}
for precipitous;
Jech and others have used in fact very similar
generalizations. (E.g., in~\cite{MR739911} it is shown more or less
that pseudo-precipitous ideals are $\infty$-semi precipitous.)

\begin{proof}
  We assume that there is a forcing $P$ and a name $\name D$ for
  the $V$-generic filter. In particular:
  \begin{equation}\label{eq:gu3}
    \parbox{0.8\textwidth}{
    $P$ forces that in $V[G]$ there is an elementary
    embedding $j:V\to M$ for some transitive class $M$
    in $V[G]$.}
  \end{equation}

  If we are only interested in consistency strength, we can use
  Dodd-Jensen core model theory as a black-box: \eqref{eq:gu3} is equiconsistent
  to a measurable cardinal,  which follows immediately, e.g., from 
  \cite[35.6]{MR1940513} and the remark after 
  \cite[35.14]{MR1940513}: $K^V=K^{V[G]}$, and there is
  a measurable iff there is an elementary embedding $j:K\to M$ (which also
  implies $M=K$). However, this only tells us that there is some ordinal
  which is measurable in an inner model, and not that this ordinal is indeed $\kappa$.

  To see this, we can either use more elaborate core model theory
  (as pointed out by Gitik, cf.~\cite[7.4.8,7.4.11]{MR1876087}).
  Alternatively, we can just slightly modify
  the proof of~\cite[Theorem 2]{MR560220}
  (which can also be found in \cite[22.33]{MR1940513}). We will do
  that in the following:
  Let $K$ be the class of strong limit cardinals $\mu$ such that
  $\cf(\mu)>\kappa$ and $\mu>|P|$. Let $(\gamma_n)_{n\in\omega}$ be an increasing sequence in
  $K$ such that $|K\cap \gamma_n|=\gamma_n$. Set $A=\{\gamma_n:\, n\in\omega\}$
  and $\lambda=\sup(A)$.

  By a result of Kunen, it is enough to show the following:
  \begin{equation}\label{eq:gu}
    \parbox{0.8\textwidth}{
    There is (in $V$) an iterable, normal, fine $L[A]$-ultrafilter $W$
    such that every iterated ultrapower is wellfounded.}
  \end{equation}

  We have a name $\name D$ for the $V$-generic filter.
  $\name D$ does not have to be normal, but
  there is some $p_0\in P$ and
  $\alpha_0\geq \kappa$ such that $p_0$ forces
  that $[\rm{Id}]=\alpha_0$.
  We set
  \begin{gather*}
    \mathcal{J}=\{x\subseteq \kappa:\, p_0\forc x\notin \name D\},\text{ and}\\
    U=\{x\in \pow(\kappa)\cap L[A]:\, x\notin \mathcal{J} \}.
  \end{gather*}
  $U$ is generally not normal, but the normalized version of $U$
  will be as required.

  {\em $U$ is an $L[A]$ ultrafilter:}
  Let $x\subseteq \kappa$ be in $L[A]$. We have to show:
  $x$ or $\kappa\setminus x$ are in $\mathcal{J}$.
  \begin{itemize}
    \item
      There is a formula $\varphi$
      and a finite
      $E\subseteq \kappa\cup K$ such that  (in  $L[A]$) $\alpha\in x$
      iff $\alpha<\kappa$ and  $\varphi(\alpha,E,A)$.
    \item Assume $G$ is $P$-generic over $V$ and contains $p_0$.
      $[\rm{Id}]=\alpha_0$, so
      $x\in \name D[G]$ iff $\alpha_0\in j(x)$.
    \item
      By elementarity
      (in $V[G]$)
      $\alpha_0\in j(x)$ iff
      $j(L[A])$ thinks that  $\varphi(\alpha_0,j(E),j(A))$.
      But $j(\mu)=\mu$ for every $\mu \in K$.
    \item
      So we get $x\in \name D[G]$ iff (in $L[A]$)
      $\varphi(\alpha_0,E,A)$ holds, independently of $G$ (provided
      $G$ contains $p_0$). In other words, if there is
      some generic $G$ such that $x\in \name D[G]$, then
      $x\in \name D[G]$ for all generic $G$ (containing $p_0$);
      i.e.\ $p_0$ forces that $x\in \name D[G]$; i.e.\
      $\kappa\setminus x\in \mathcal J$.
    \item
      Assume that $x$ is not in $\mathcal{J}$.
      Then there is some $q\leq p_0$ forcing
      that $x\in \name D$. So $\kappa\setminus x\in \mathcal J$.
  \end{itemize}

  {\em $U$ is ${<}\kappa$-complete, fine and wellfounded:}
    Pick $\lambda<\kappa$ and  $(x_\alpha)_{\alpha\in\lambda}$
    in $L[A]$ such that each $x_\alpha\in U$.
    Then $p_0$ forces that $\kappa\setminus x_\alpha\notin \name D$,
    and therefore that $\bigcup \kappa\setminus x_\alpha \notin \name D$
    (since $\name D$ is a ${<}\kappa$-complete ultrafilter).
\\
    This also shows that
    (in $V$) the intersection of $\al0$ many $U$-elements
    is nonempty; which implies that every iterated ultrapower
    is wellfounded (provided iterability).

  {\em $U$ is iterable:}
  Let (in $L[A]$) $(x_\alpha)_{\alpha\in\kappa}$ be
  a sequence of subsets of $\kappa$.
  Let $G$ be $P$-generic over $V$ and contain $p_0$.
  In $V[G]$,
  $x_\alpha\in \name D[G]$ iff $\alpha_0\in j(x_\alpha)$.
  The sequence $(j(x_\alpha))_{\alpha\in\kappa}$ is
  in $L[j(A)]$, and therefore also the set
  $\{\alpha\in\kappa:\, \alpha_0\in j_G(x_\alpha)\}$.
  But $L[j(A)]=L[A]$.

  {\em normalizing:}
  Since we now know that $U$ is wellfounded,
  we know that there is some
  $f:\kappa\to\kappa$ in $L[A]$  representing
  $\kappa$.
  Set $W=f_*(U)$. Then $W$ is as required.
\end{proof}

The following follows easily from Kunen's method of iterated ultrapowers
(see, e.g.,~\cite[4.3]{MR2371208} for a proof):
\begin{Lem}\label{lem:gurke3}
  Assume $V=L[U]$, where $U$ is a normal ultrafilter on $\kappa$.
  Let $V'$ be a forcing extension of $V$ and $D\in V'$ a normal,
  wellfounded $V$-ultrafilter on $\kappa$. Then $D=U$.
\end{Lem}
This implies:
\begin{Cor}\label{cor:gurke89}
  In $L[U]$, the dual of $U$ is the only normal precipitous ideal
  on $\kappa$; and every ideal on $\kappa$
  that is normally $\infty$-semi precipitous
  is a subideal of the dual of $U$.
\end{Cor}

We will also need the following:
\begin{Lem}\label{lem:nowhereprec}
  If $I$ is a ${<}\kappa$-complete ideal,
  $P$ a $\kappa$-cc forcing notion, and 
  $\cl(I)$ the $P$-name for the closure of $I$ in $V[G]$,
  then $P$ preserves the following properties:
  $I$ is precipitous, $I$ is not precipitous, and $I$ is nowhere
  precipitous.
\end{Lem}

\begin{proof}
   This has been known for a long time, cf.\ e.g.~\cite{MR713297}:
   ``not precipitous'' is equivalent to the existence of a 
   decreasing sequence of functionals starting at some positive set
   $S_0$ (this corresponds to: $S_0$ forces that there is an infinite
   decreasing sequence in the ultrapower, the sequence of functionals
   witnesses this). A $\kappa$-cc forcing preserves maximality (below $S_0$)
   of an antichain in $B_I$,
   and therefore the decreasing sequence of functionals.
   ``Nowhere precipitous'' is equivalent to the existence 
   of a decreasing sequence of functionals starting with $\kappa$,
   which again is preserved by $P$.
\end{proof}

\section{Nonempty winning}
Let us  assume that nonempty has a winning strategy in $\PDzne(I)$
(or a similar game such as $\PDIne(I)$).
A valid sequence is a finite initial sequence of  a run of
the game $\PDzne$,
where nonempty uses his strategy.  So a valid sequence $w$ has the form
$(f_0,\alpha_0,f_1,\alpha_1,\dots,f_{n-1},\alpha_{n-1})$,
where $f_i$ is a regressive function and $\alpha_i$
the value chosen by the strategy. In particular
$S_i=\bigcap_{j\leq i}f_j^{-1}(\alpha_j)$ is $I$-positive
for each $i<n$.
We set
\[
  A(w)=S_{n-1}=\bigcap_{j<n}f_j^{-1}(\alpha_j).
\]

\begin{Def}
  $P^*$ is the set of valid sequences ordered by extension.
  (A longer sequence is stronger, i.e., smaller in the $P^*$-order.)
\end{Def}

So if $w< v$, then $A(w)\subseteq A(v)$.
If $w_0>w_1>w_2>\dots$ is
an infinite decreasing sequence in $P^*$,
then $\bigcup_{i\in\omega} w_i$
represents a run of the game, so the result
$\bigcap_{i\in\omega} A(w_i)$ has to be nonempty (or even positive
in the case of a $\PDIne$-strategy).

\begin{Lem}\label{lem:bPDgtc}
  $\propb(\PDzne)$ implies $\kappa>2^{\al0}$.
\end{Lem}

Actually, we can even restrict nonempty to play
functions $f:\kappa\to \{0,1\}$.
In other words,
it is enough to assume $\propb(\Gmin(I,2))$,
cf.\ Definition~\ref{def:Gmin}.
\begin{proof}
  The proof is the same as \cite[\S1]{MR0485391}:
  We assume otherwise and identify $\kappa$ with a subset
  $X$ of $[0,1]$ without a perfect subset.
  We claim:
  \begin{equation}\label{eq:gu2}
    \parbox{0.8\textwidth}{
      For all $w\in P^*$ and $n\in \omega$ there are
      disjoint open intervals $I_1$ and $I_2$ of length
      ${\leq} 1/n$ and
      $w_1,w_2<w$ such that $A(w_1)\subseteq I_1$
      and $A(w_2)\subseteq I_2$.}
  \end{equation}
  Assume that \eqref{eq:gu2} fails for some $v_0$ and $n_0$.
  Given $v<v_0$ and $n>n_0$, we fix a partition of
  $[0,1]$ into $n$ many open intervals of length $1/n$
  and the (finite) set of endpoints.
  By splitting $A(v)$ $n+1$ many times,
  empty can guarantee that $A(w)$ has to be subset of
  one of the intervals for some $w<v$.
  Since \eqref{eq:gu2} fails,
  there has to be for each $n$ a fixed element $I(n)$ of the
  partition
  such that for all $v<v_0$ there is a $w<v$
  with $A(w)\subseteq I(n)$. $\bigcap I(n)$
  can contain at most one point $x$, so the empty
  player can continue $v_0$ by first splitting
  into $\{x\}$ and $A(v_0)\setminus \{x\}$;
  and then extending each $v_{n-1}$ to $v_n$ such that
  $A(v_n)\subseteq I(n)$. Then the intersection is empty.
  This shows \eqref{eq:gu2}.

  So we  can fix an order preserving function $\psi$ from $2^{<\omega}$
  to $P^*$ such that $A(\psi(s^\frown 0))$
  and $A(\psi(s^\frown 1))$ are separated by intervals of
  length ${\leq} 1/{|s|}$
  for all $s\in 2^{<\omega}$.
  Then every $\eta\in 2^{\omega}$ is mapped
  to a run of the game,
  and since nonempty wins, there is some
  $r_\eta\in\bigcap_{n\in\omega }A(\psi(\eta\restriction n))$.
  This defines a continuous, injective mapping 
  from $2^\omega$ into $X$ and therefore a perfect subset of $X$.
\end{proof}

Clearly $\propa(\PDIne)$ fails if $I$ is concentrated on $E^\kappa_{\al0}$,
and this was used in \cite{MR0485391}
to show that in this case $\propb(\BMzne(I))$ fails as well.
A similarly easy proof gives:
\begin{Lem}\label{lem:uzligutzli}
  $\propb(\PDzne(I))$ fails if
  $I$ is concentrated on $E^\kappa_{\al0}$.
\end{Lem}

\begin{proof}
  Assume otherwise.
  Fix for each $\alpha\in E^\kappa_{\al0}$ a normal cofinal
  sequence $(\seq(\alpha,n))_{n\in\omega}$, and let $g_i:\kappa\to\kappa$
  map $\alpha$ to $\seq(\alpha,i)$.
  We first show a variant of \eqref{eq:gu2}:
  \begin{equation}\label{eq:lkwet}
     \text{For all $w$ there are $v_1,v_2\leq w$ in $P^*$
       such that $A(v_1)\cap A(v_2)=\emptyset$.}
  \end{equation}
  Assume otherwise. Then for each $i$ there is a fixed $\beta_i$ such that
  nonempty responds with $\beta_i$ whenever empty plays $g_i$
  in any $v\leq w$. Set $\delta=\sup\{\beta_i:\, i\in\omega\}$,
  and let empty play the following response to $w$:
  \[
     f(\alpha)=
     \begin{cases}
         0
       &
         \text{if }\alpha\leq \delta,
       \\
         \min\{n:\, \seq(\alpha,n)>\delta\}
       &
         \text{otherwise}.
     \end{cases}
  \]
  If nonempty responds to $f$ with $m$, then
  empty can play $g_m$ as next move, nonempty has to respond
  with $\beta_m<\delta$, but
  \[
    g_m^{-1}(\beta_m)=\{\alpha:\, \seq(\alpha,m)=\beta_m\}
  \]
  is disjoint to $f^{-1}(m)$, a contradiction.
  This shows \eqref{eq:lkwet}.

  Now fix $N\esm H(\chi)$ of size less than $\kappa$
  containing the strategy as well as all $g_n$
  and
  such that $N\cap \kappa=\delta\in E^\kappa_{\al0}$.
  We define a sequence $w_0>w_1>\cdots$ in $P^*$
  such that each $w_i$ is in $N$:
  Using \eqref{eq:lkwet} in $N$, we get a
  $w_0\in N\cap P^*$ such that $\delta\notin A(w_0)$.
  Given $w_{n-1}$, let $w_n\in N$ be the continuation
  where empty played the regressive function
  \[
     f_n(\alpha)=
     \begin{cases}
       0& \text{if }\alpha< \seq(\delta,n)\\
       g_n(\alpha)&\text{otherwise}.
     \end{cases}
  \]
  (Note that $\seq(\delta,n)<\delta$ is in $N$ for all $n$.)
  Assume that $\nu\in\bigcap_{n\in\omega} A(w_n)$. Then
  $\nu\geq \seq(\delta,n)$ for all $n$, so $\nu\geq \delta$.
  On the other hand, $g_n(\nu)\in N$ for all $n$, so $\nu\leq \delta$.
  But $\delta\notin A(w_0)$, a contradiction.
\end{proof}
Of course this shows the following: $\propb(\PDzne(I))$ implies
$\propb(\PDzne(I\restriction E^\kappa_{>\al0}))$ (since empty can just cut
$\kappa$ into $E^\kappa_{\al0}$ and $E^\kappa_{>\al0}$ as a first move).

Recall that $\propb(\PDzne(I))$ for any $I$ implies
$\propb(\PDzne(\NS))$ (due to monotonicity).
So the last lemma gives:
\begin{Cor}\label{cor:gubadi}
  $\propb(\PDzne(I))$ is equivalent to
  $\propb(\PDzne(I\restriction E^\kappa_{{>} \al0}))$ and
  implies $\propb(\PDzne(\NS))$  and
  $\propb(\PDzne(\NS\restriction E^\kappa_{{>} \al0}))$.
\end{Cor}

\begin{Lem}\label{lem:bPDgne}
  $\propb(\PDzne(I))$ implies that $I$ is normally $\infty$-semi precipitous.
  \\
  $\propb(\Gmin(I,{<}\kappa))$ implies that $I$ is $\infty$-semi precipitous.
\end{Lem}

\begin{proof}
  We define the $P^*$-name $\name U$ by $X\in \name U$ iff $X\supseteq A(w)$ for
  some $w\in G_{P^*}$.
  \begin{itemize}
    \item $P^*$ forces that $\name U$ is a $V$-ultrafilter:
      Given any $w\in P^*$ and
      $X\in V$, player empty can respond to $w$
      by cutting into $X$ and $A(w)\setminus X$.
    \item In the $\Gmin$ case,
      $P^*$ forces that $\name U$ is ${<}\kappa$-complete:
      Assume that (in $V$) $X$ is the disjoint union of
      $(X_i)_{i\in \lambda}$, $\lambda<\kappa$. Then empty
      can responds to $w$ by cutting into 
      $\{X_i:\,i\in \lambda\}\cup \{A(w)\setminus X\}$.
    \item In the case of $\PDIne$,
      $P^*$ forces that $\name U$ is $V$-normal:
      If $f\in V$ is regressive, then empty can play
      $f$ as response to any $w$.
    \item $P^*$ forces that $\name U$ is wellfounded:
      Assume towards a contradiction that
      $w$ forces that $(\name{f}_n)_{n\in\omega}$
      are functions (in $V$) from $\kappa$ to the ordinals such that
      \[
        \name{A}_n=\{\alpha:\, \name{f}_{n+1}(\alpha)<\name{f}_{n}(\alpha)\}
      \]
      is in $\name U$ for all $n\in\omega$.
      Set $w_{-1}=w$.
      Assume that we already have $w_n$
      (for $n\geq -1$). Pick some $w'_{n+1}<w_n$
      deciding
      $\name{f}_{n+1}$ to be some $f'_{n+1}\in V$.
      So $w'_{n+1}$
      forces that $X_{n+1}:=\bigcap_{l\leq n+1} \name{A}_l
      =\bigcap_{l\leq n+1} A'_l$ (a set in $V$)
      is in $\name U$. In particular, there is
      some $w_{n+1}$ stronger than
      $w'_{n+1}$ such
      that $A(w_{n+1})\subseteq X_{n+1}$.
      The sequence
      $(w_n)_{n\in\omega}$ corresponds to
      a run of the game. Since nonempty follows the strategy,
      there is some $\alpha\in\bigcap_{n\in\omega} A(w_n)$.
      $w_{n+1}$ forces $\alpha\in X_{n+1}$, i.e.,
      $f'_{n+1}(\alpha)<f'_{n}(\alpha)$.
      This gives an infinite decreasing sequence, a contradiction.
      \qedhere
  \end{itemize}
\end{proof}

Together with \ref{cor:gurke89}, we get:
\begin{Cor}\label{cor:gurke90}
  In $L[U]$, nonempty does not win $\PDzne(\NS_\kappa\restriction S)$
  for any $S\notin U$. In particular,
  $\propb(\PDIne(\NS_\kappa))$ holds (even for the game of length $\kappa$),
  but $\propb(\PDze(\NS_\kappa\restriction S))$ fails for every stationary $S$.
  Also, and $\propa(\BMze(\NS_\kappa\restriction S))$ fails, i.e.,
  $\NS_\kappa$ is nowhere precipitous.
\end{Cor}

We can use a Levy Collapse to reflect this situation down to, e.g., $\al2$.
We first list some properties of the Levy collapse.
Assume that $\kappa$ is inaccessible, $\theta<\kappa$ regular, and
let $Q = \textrm{Levy}(\theta, <\kappa)$ be
the Levy collapse of $\kappa$ to $\theta^+$:
A condition $q \in Q$ is a function defined on a subset of $\kappa \times \theta$, such that
$|\dom(q)| < \theta$ and $q(\alpha, \xi) < \alpha$ for $\alpha>1, (\alpha, \xi) \in \dom(q)$
and $q(\alpha, \xi)=0$ for $\alpha\in\{0,1\}$.
Given $\alpha < \kappa$, define
$Q_\alpha = \{q:\, \dom(q) \subseteq \alpha \times \theta\}$ and
$\pi_\alpha: Q \to Q_\alpha$ by
$q \mapsto q \upharpoonright (\alpha \times \theta)$.
The following is well known:
\begin{itemize}
  \item If $q\Vdash p\in G$, then $q\leq p$ (i.e. $\leq^*$ is the same as $\leq$).
  \item $Q$ is $\kappa$-cc and $<\theta$-closed.
  \item In particular, if $p$ forces that $C\subseteq \kappa$ is club, then
    there is a club $C_0\in V$ such that $p$ forces 
    $C_0\subseteq C$. The ideal generated by $\NS^V_\kappa$
    in $V[G]$ is $NS^{V[G]}_\kappa$.
\end{itemize}

We also need the following simple fact (see, e.g.,~\cite[6.2]{MR2371208}
for a proof):
\begin{equation}\label{eq:qwkjrqwr}
  \parbox{0.8\textwidth}{
  Let $I$ be a normal ideal concentrated on $E^\kappa_{{\geq} \theta}$, 
  let $T$ be $I$-positive, $p\in Q$ and 
  $p_\alpha\leq p$ for all $\alpha\in T$. Then
  there is an $I$-positive $T'\subseteq T$ and a $q\leq p$
  such that $\pi_\alpha(p_\alpha)=q$ for all $\alpha\in T'$.
  }
\end{equation}
So in particular, every $q'\leq q$ is compatible with
$p_\alpha$ for all but boundedly many $\alpha\in T'$.

We will also use:
\begin{Lem}\label{lem:ez4636}
  Let $\kappa$ be inaccessible and $T\subset \kappa$ be stationary.
  The Levy collapse preserves
  $\lnot \propb(\PDzne(\NS_\kappa\restriction T))$. The same
  holds for $\PDIne$.
\end{Lem}
\begin{proof}
  Assume towards a contradiction that $q$ forces that nonempty does have a
  winning strategy in $V[G]$. We describe a winning strategy
  in $V$: Assume empty plays $f_0$ (in $[V]$).  Let $q_0\leq q$ decide that in
  $V[G]$ nonempty chooses $\alpha_0$ as response to $f_0$ according to the
  winning strategy in $V[G]$.
  So  $q_0$ forces that $f_0^{-1}(\alpha_0)\cap T$ is stationary,
  therefore $f_0^{-1}(\alpha_0)\cap T$ is stationary in $V$.
  Generally, let $q_n\leq q_{n-1}$ decide that nonempty
  plays $\alpha_n$ as response to $f_n$.
  Since $Q$ is $\sigma$-closed, there is a $q_\omega<q_n$
  for all $n$. So $q_\omega$ forces that
  $\bigcap f_n^{-1}(\alpha_n)\cap T$ is stationary.
\end{proof}

Starting with $L[U]$ and using a Levy collapse we get:
\begin{Cor}\label{cor:gurk422}
  Consistently relative to a measurable,
  $\propb(\PDIne(\NS_{\al2}))$ holds (even for length $\al1$) but
  $\propb(\PDze(\NS_{\al2}\restriction S))$
  fails for every stationary $S$, and
  $\NS_{\al2}$ is nowhere precipitous.
\end{Cor}

\begin{proof}
  Assume $V=L[U]$ and let
  $Q=\textrm{Levy}(\al1,{<}\kappa)$
  be the Levy collapse of $\kappa$ to $\al2$.

  To see that $\NS_{\al2}$ is forced to be nowhere precipitous, 
  note that ${<}\kappa$-cc implies $\cl^{V[G]}(\NS^V_\kappa)=\NS^{V[G]}_\kappa$ 
  and use \ref{lem:nowhereprec}.

  In $V[G]$, $\cl^{V[G]}(U)$ is a normal filter such that the family
  of positive sets has a $\sigma$-closed
  dense subset \cite{MR0485391}. Let $I$ be the dual ideal. 
  So nonempty wins $\BMIe(I)$, and therefore $\PDIe(I)$ and $\PDIne(\NS_\kappa)$
  (even of length $\al1$).

  It remains to be shown that $\propb(\PDze(\NS_{\al2}\restriction S))$ fails
  in $V[G]$ for all stationary $S$.
  Assume towards a contradiction that some $p$
  forces that $\name S$ is stationary and
  $\propb(\PDzne(\NS_\kappa\restriction S'))$ holds
  for all stationary $S'\subseteq \name S$.
  According to~\ref{cor:gubadi}
  we can assume $\name S\subseteq E^{\al2}_{\al1}$.
  Set
  \[
    T_0=\{\alpha\in\kappa:\, p\not\forc \alpha\notin \name S\}
  \]
  $T_0\subseteq E^\kappa_{{\geq} \al1}$
  is stationary. Fix some stationary $T\subseteq T_0$
  not in $U$; and for $\alpha\in T$ fix some $p_\alpha\leq p$ forcing
  $\alpha\in \n S$.
  Apply~\eqref{eq:qwkjrqwr} to $T$, the nonstationary ideal
  and $(p_\alpha)_{\alpha\in T}$. This results in $q\leq p$
  and $T'\subseteq T$ stationary.
  \begin{equation}\label{eq:gkj3}
     q\forc S':=T'\cap \name S\text{ is stationary}.
  \end{equation}
  Otherwise some $q_1\leq q$ forces that $S'$
  is nonstationary. Then there is in $V$ a club $C$ and a $q_2\leq q_1$
  forcing that $S'\cap C=\emptyset$. Pick $\alpha\in T'\cap C$
  such that $p_\alpha$ and $q_2$ are compatible. Then
  $q_3\leq p_\alpha,q_2$ forces that $\alpha\in T'\cap C\cap \n S$,
  a contradiction. This shows~\eqref{eq:gkj3}.

  By our assumption, $p$ forces that nonempty wins $\PDzne(\NS\restriction S')$.
  But $\propb(\PDzne(\NS_\kappa\restriction T'))$ fails in $V$
  (since $T'\subset T$ and $T\notin U$), 
  therefore $\propb(\PDzne(\NS_{\al2}\restriction T'))$
  fails in $V[G]$ according to~\ref{lem:ez4636}, and by monotonicity
  $\propb(\PDzne(\NS_{\al2}\restriction S'))$ fails as well,
  a contradiction.
\end{proof}

We will now force nonempty not to win PD. For simplicity we will assume 
CH and look at $\kappa=\al2$. It turns out that it is enough to 
add $\al1$ many Cohen reals (actually, many similar forcings also work).
First we need another variant of~\eqref{eq:gu2} or~\eqref{eq:lkwet}:
\begin{Lem}\label{lem:prelim1}
  Assume CH and $\propb(\PDzne(\NS_{\al2}))$.
  For each $v\in P^*$ there are
  $F'(v)\leq v$ and $F''(v)\leq v$ such that $A(F'(v))$ and
  $A(F''(v))$ are disjoint.
\end{Lem}
(We can choose $F'(v)$ and $F''(v)$ to be immediate successors
of $v$, i.e. we just have to choose two regressive functions
$f'$ and $f''$ as empty's moves.)

\begin{proof}
  We fix an injection $\phi:[\al2]^{\al0}\to \al2$.
  Let $S=C\cap  E^{\al2}_{\om1}$ (for some clubset $C$) consist of
  ordinals $\alpha$ such that $\phi''[\alpha]^{\al0}\subseteq \alpha$.
  For each $\alpha\in S$, pick a normal cofinal
  sequence $\gamma^\alpha:\om1\to\alpha$.
  For $i\in\omega_1$ set $g_i(\alpha)=\phi(\{\gamma^\alpha(j):\, j\leq i\})$
  for $\alpha\in S$;
  and  set $g_i(\alpha)=0$ for $\alpha\notin S$.
  So for all $i\in\om1$, $g_i$ is a regressive function.
  If $\alpha\neq \beta$ then $g_i(\alpha)\neq g_i(\beta)$
  for some $i$; and
  $g_i(\alpha)\neq g_i(\beta)$ implies
  $g_j(\alpha)\neq g_j(\beta)$  for all $j>i$.

  Let $x(i)$ be the strategy's response to $v^\frown g_i$,
  We can identify $x(i)$ with the sequence
  $\phi^{-1}x(i)=(\gamma_{i,k})_{k\leq i}$. 
  So for all $\alpha$  with $g_i(\alpha)=x(i)$ we get
  $\gamma^\alpha(k)=\gamma_{i,k}$ for $k\leq i$.

  Case A: There are $k<i<j<\omega_1$
  such that $\gamma_{i,k}\neq \gamma_{j,k}$.
  Then set $F'=g_i$
  and $F''=g_j$. If $\alpha\in g_i^{-1}(x(i))$ and
  $\beta\in g_j^{-1}(x(j))$, then
  $\gamma^\alpha(k)\neq \gamma^\beta(k)$, so in particular
  $\alpha\neq \beta$.
  
  Case B: Otherwise, all the sequences
  $(\gamma_{i,k})_{i\leq k}$ cohere for all $i\in\omega_1$,
  so let $(\tilde \gamma_k)_{k\in\omega_1}$ be the union of these
  sequences, with supremum $\tilde \alpha<\omega_2$.
  So for all $\alpha\neq \tilde \alpha$ in $S$ there is some
  $k(\alpha)\in\omega_1$ such that $\gamma^\alpha(k(\alpha))\neq \tilde \gamma_{k(\alpha))}$. Set $k(\alpha)=0$ for $\alpha\in \omega_1\cup\{\tilde \alpha\}
  \cup \omega_2\setminus S$. So $k$ is a regressive function. Let
  $l$ be the strategy's response to $v^\frown k$.
  Set $F'=k$ and $F''=g_{l}$. If $\alpha\in k^{-1}(l)$ and
  $\beta\in g_l^{-1}x(l)$, then $\gamma^\beta(l)=\tilde \gamma_l$
  which is different to $\gamma^\alpha(l)$.
\end{proof}

\begin{Lem}\label{lem:fueel}
  Assume CH. Let $P_{\om1}$ be the forcing notion adding
  $\al1$ many Cohen reals.
  Then $P_{\om1}$ forces $\lnot\propb(\PDzne(\NS_{\al2}))$.
\end{Lem}
(The same holds for any other CH preserving $\om1$-iteration of absolute ccc
forcing notions.) Note that since $\PDzne$ is monotone, $\propb(\PDzne(I))$
fails for all ideals $I$ on $\al2$.

\begin{proof}
  Assume that $p\in P_{\om1}$ forces that $\n\tau$ is a winning strategy
  for nonempty for the game $\PDzne(I)$.

  Let $P_\alpha$ be the 
  complete subforcing of the first $\alpha$ Cohen reals.
  $P_{\om1}$ forces that
  Lemma~\ref{lem:prelim1} holds. We fix the according
  $P_{\om1}$-names $\n F'$ and $\n F''$.
  Let $N\esm H(\chi)$ be countable and contain $p$, $\n\tau$, $\n F'$ and $\n F''$.
  Set $\epsilon=N\cap \omega_1$.
  If $G_{\om1}$ is $P_{\om1}$-generic over $V$,
  then $G_{\epsilon}=G_{\om1}\cap P_{\epsilon}$
  is $P_{\omega_1}$-generic over $N$ (and
  $P_\epsilon$-generic over $V$).

  So in $N_\epsilon=N[G_{\epsilon}]=N[G_{\om1}]$, we can
  evaluate the correct values of $\n\tau$, $\n F'$ and
  $\n F''$ for all valid sequences $v$ in $N_\epsilon$
  (i.e., the resulting values are the same as the ones
  calculated in $V_{\om1}=V[G_{\om1}]$).

  In $V_{\om1}$, pick any real $r\notin V_\epsilon$.
  Using $r$, we now define by induction a run $b$ of the game
  such that each initial segment is in $N_\epsilon$:
  Assume we already have the valid sequence $u\in N_\epsilon$.
  Extend $u$ with $\n F'(u)$ if $r(n)=0$, and to
  $F''(u)$ otherwise.

  So $b\in V_{\om1}$ is a run of the game according to $\tau$;
  nonempty wins the run; so there is some
  $\delta\in \bigcap_{n\in\omega}A(b\restriction n)$.  
  But we can in $V_\epsilon$ use
  this $\delta$ to reconstruct (by induction)
  the run $b$ and therefore the real $r$:
  Assume we already know $r\restriction n$
  and the corresponding valid sequence $u=b\restriction n$.
  Then $\delta$ is element of exactly
  one of $A(F'(u))$ or $A(F''(u))$, which
  determines $r(n)$ as well as the sequence corresponding
  to $b\restriction (n+1)$.
\end{proof}

On the other hand, adding Cohens, as any $\kappa$-cc forcing, preserves
precipitousness (and non-precipitousness) of an ideal,
cf.~\ref{lem:nowhereprec}.  So we get:
\begin{Cor}\label{cor:funfz}
  $\propa(\BMze(I))$ does not imply $\propb(\PDzne(\NS))$.
\end{Cor}

If we assume CH and an $\al3$-saturated normal ideal on $\al2$
saturated on $E^{\al2}_{\al1}$, we get the following:
\begin{Cor}\label{cor:funfe} (Saturated ideal.)
  $\propa(\BMIe(I))$ does not imply $\propb(\PDzne(\NS))$.
\end{Cor}
\begin{proof}
  Since $P_{\om1}$ has size $\al1<\al2$, $\cl(I)$ remains $\al3$-saturated.
  So in $V[G]$, we can use \ref{cor:wrqw} to see that
  $\propa(\BMIe(\cl(I)))$ holds.
\end{proof}

\section{$\mathfrak b(\textrm{\textup{PD}}_{\textrm{\textup{e}}})$ for a nonprecipitous $I$.}

We have seen that $\propb(\PDIne(I))$ can hold for a nowhere precipitous
ideal $I$.
It is a bit harder to show that there can be a nowhere precipitous
ideal $I$ that even satisfies $\propb(\PDIe(I))$.

\begin{Fact}\label{thm:oiuqweh}
  The following is consistent relative to $\kappa$ measurable:
  $I_0$ is nowhere precipitous, 
  and for every $I_0$-positive set $S$ the dual to
  $I_0\restriction S$ can be extended to a normal ultrafilter.
\end{Fact}
Note that this implies $\propb(\PDIe(I_0))$, even for the game of length
$\kappa$.

And as usual, we can use a Levy collapse to reflect these properties to $\al2$:
\begin{Lem} Start with a universe $V$ as in Fact~\ref{thm:oiuqweh}.
  After collapsing $\kappa$ to $\al2$, we get:
  $\cl(I)$ is nowhere precipitous and satisfies 
  $\propb(\PDIe(\cl(I)))$ (even for the game of length $\al1$).
\end{Lem}
\begin{proof}
  Nowhere precipitous follows from \ref{lem:nowhereprec}.
  Let $S$ be a $P$-name for a $\cl(I)$-positive set
  and $p\in P$.
  Will show:
  \begin{equation}\label{eq:wtrqtwr}
    \parbox{0.8\textwidth}{%
      In $V$ there is a normal ultrafilter 
      $U$ and a $q\leq p$ forcing that $S$ is $\cl(U)$-positive.}
  \end{equation}
  Then according to the usual argument,
  the $\cl(U)$-positive sets have a $\sigma$-closed 
  dense subset, so nonempty wins $\BMIe(\cl(U)\restriction S)$,
  and --- since $\cl(U)$ extends $\cl(I)$ --- nonempty
  wins $\PDIne(I\restriction S)$ (even for length $\al1$).

  To prove  \eqref{eq:wtrqtwr}, 
  set $T=\{\alpha\in E^\kappa_{{\geq} \al1}:\, p\not\forc \alpha\notin
  S\}$. $T$ is $I$-positive.
  For each $\alpha\in T$ pick a witness $p_\alpha\leq p$.
  Let $q,T''$ be as in~\eqref{eq:qwkjrqwr}
  and pick a normal ultrafilter $U$ containing $T''$. We have
  to show that $q$ forces $S$ to be $\cl(U)$-positive.
  Assume otherwise, and pick $q'\leq q$ and $A\in U$ such that 
  $q'$ forces $A\cap S=\emptyset$. Then 
  $q'\in Q_\alpha$ for some $\alpha<\kappa$.
  Pick $\beta\in T''\cap A\setminus \alpha$. Then $p_\beta$ and $q'$
  are compatible, a contradiction to $p_\beta\forc \beta\in S$.
\end{proof}

After this paper was submitted, it came to our attention that
Fact~\ref{thm:oiuqweh} follows directly from a construction of Gitik, using
only a measurable:
In his paper {\em Some pathological examples of precipitous
ideals}~\cite{MR2414461}, he constructs a non-precipitous filter $U^*$ as
intersection of normal ultrafilters (see page 502 and Lemma~3.3).

We still give our proof of Fact~\ref{thm:oiuqweh} in the rest of the paper,
using a supercompact and assuming GCH in the ground model, since the
construction itself might be of some interest.

We will split the proof into several lemmas:
First we define the forcing $S(\kappa)$ as limit of $P_\alpha$.
We also define dense subsets $P'_\alpha$ of the $P_\alpha$.
Then we define the forcing notion $R_{\kappa+1}$, by doing the usual
Silver-style preparation with reverse Easton support. This
forcing notion is as required: In the
extension, we define~\ref{def:I} the ideal $I_0$ and show
that Fact~\ref{thm:oiuqweh}  holds
(\ref{lem:yrtu}~and~\ref{lem:hjwte}).

\subsection{The basic forcing}

So let us assume that $\kappa$ is an inaccessible cardinal, and define
$S(\kappa)$  as the limit of the ${<}\kappa$-support iteration
$(P_a,Q_a)_{a\in\kappa^+}$ of length $\kappa^+$ defined the following way: 
By induction on $a$, we define $Q_a$ together with
the $P_a$-names $B_a\subseteq \kappa$,
$g_a:\kappa\to\kappa+1$ and the $P_{a+1}$-names $A_a\subseteq \kappa$,
$f_a:\kappa\to\kappa$:

We identify the tree $T=(\kappa^+)^{{<}\omega}$ of finite $\kappa^+$-sequences
with $\kappa^+$ such that the root is identified with $0$.  We can assume
that $a<_T
b$ implies $a<b$ (as ordinals in $\kappa^+$).  We write $a\lhd_T b$ 
or $b\rhd_T a$ to
denote that $b$ is immediate $T$-successor of $a$. 
So for all $a\in\kappa^+$
there are $\kappa^+$ many $b$ with $a\lhd_T b$.  For $b\neq 0$ we also write
$\predec(b)$ to denote the (unique)  $a$ such that $a\lhd_T b$.

Assume we already have defined 
$P_a$, and the $P_{b+1}$-names $A_b,f_b$ for all $b<a$.
Then in $V[G_{P_a}]$, we define $B_a$, $g_a$, $Q_a$ 
and the $Q_a$-names $f_a$, $A_a$:
\begin{itemize}
  \item If $a=0$, we set $g_a(\alpha)=\kappa$ for all $\alpha\in\kappa$,
    and $B_a=\kappa$.
  \item Otherwise, we use some bookkeeping%
    \footnote{
      We just need to guarantee that $P_{\kappa^+}$ forces:
      For every $a\in T$ and every subset $B$ of $A_a$ 
      there is a $b\rhd_T a$ 
      such that $B^0_b=B$. Note that $A_b\subseteq B
      \subseteq A_a$.%
    }
    to find
    a $B^0_a\subseteq A_{\predec(a)}$, and we set:
    \begin{equation}\label{eq:wqoiu2}
      B_a=B^0_a\setminus \nabla_{b<a:\, \predec(b)=\predec(a)}A_b,
      \text{ and we set }
      g_a=f_{\predec(a)}.
    \end{equation}
  \item A condition $p$ of $Q_a$ is a function $f^p:\beta^p\to \kappa$
    such that $\beta^p\in\kappa$ and for all $\alpha\in\beta^p$:
    \begin{itemize}
      \item 
        if $\alpha\notin B_a$ or $g_a(\alpha)=0$ then $f^p(\alpha)=0$,
      \item otherwise $f^p(\alpha)<g_a(\alpha)$. 
      \item Additionally, if $a=0$ we require
        $f^p(\alpha)>0$.
    \end{itemize}
  \item We define the order on $Q_a$ by $q\leq p$ if $f^q\supseteq f^p$. 
  \item We set $f_a$ to be the canonical $Q_a$-generic, i.e.,
    $\bigcup_{q\in G}f^q$.
  \item We set $A_a=\{\alpha\in\kappa:\, f(\alpha)>0\}$. (So $A_0=\kappa$,
    and $A_a\subseteq B_a\subseteq B^0_a$.)
\end{itemize}

Note that to write the diagonal union in~\eqref{eq:wqoiu2}, 
we have to identify the index set with $\kappa$. Different 
identifications lead to the same result modulo club. In particular,
we get:
\begin{equation}\label{eq:lublub}
  \text{If }b<a\text{ and }\predec(b)=\predec(a)\text{ then }
  B_a\cap A_b\text{ is nonstationary.} 
\end{equation}

Obviously $Q_a$ is ${<}\kappa$-closed.
We now define $P'_a$ by induction on
$a\in\kappa^+$ and show (in the same induction) that $P_a'$
is ${<}\kappa$-closed and can be interpreted
to be a dense subset of $P_a$.
A condition $p\in P'_a$ is a function from 
$u\times \beta$ to $\kappa$ such that:
\begin{itemize}
  \item $\beta\in\kappa$.
  \item $u$ is a subset of $a$ of size ${<}\kappa$.
  \item $c\lhd_T b$ implies $\max(1,p(c,\alpha))>p(b,\alpha)$.
  \item $p(b,\alpha)>0$ implies that  $p\restriction b$ forces 
    (as element of $P_b$)%
    \footnote{%
      by induction, we already know that $P'_b$ is dense in $P_b$%
    }
    that $\alpha\in B_b$.
  \item If $0\in u$, then $p(0,\alpha)>0$ for all $\alpha<\beta$.
\end{itemize}
We can interpret $p\in P'_a$ to be a condition in $P_a$ in the 
obvious way; in particular we can define the order on $P'_a$
to be the one inherited from $P_a$.

\begin{Lem}
  \begin{itemize}
    \item $P'_a$ is a dense subset of $P_a$.
    \item The order on $P'_a$ (as inherited from $P_a$) is
      the extension relation.
    \item $P'_a$ is ${<}\kappa$-closed.
    \item $P_a$ is strategically ${<}\kappa$-closed.
  \end{itemize}
\end{Lem}
\begin{proof}
  By induction on $a$ (formally, the definition of $P'_a$
  has to be done in the same induction as well).
  It is clear that $P'_a$ is closed and that the order
  is extension. We have to show that 
  $P'_a$ is dense in $P_a$. We do that  by case distinction on
  $\cf(a)$:
  \\
  {\bf The case} $\cf(a)\geq \kappa$ is trivial.
  \\
  {\bf The successor case:} Assume $a=b+1$ and $p\in
  P_a$. Then by induction we know that $P_b$ is strategically $\kappa$ closed,
  so we can strengthen $p\restriction b$ to some $p'\in P'_b$ deciding $p(b)$
  to be
  some $f^p$. We can assume that the height of $p'$ is at least the height of
  $f^p$, and we can extend $f^p$ up to the height of $p'$ by adding zeros on
  top.  Then $p'$ together with $f^p$ is a condition of $P'_a$ stronger
  than $p$.
  \\
  {\bf The case} $\cf(a)<\kappa$, i.e., 
  $a=\sup(b_i:\, i\in\lambda)$ for some $b_i<a$ and $\lambda<\kappa$. 
  We assume $p\in P_a$.
  We define by induction on $i\in \lambda$ decreasing conditions
  $p'_i\in P'_{b_i}$ stronger than $p\restriction b_i$.
  (By induction we know that $P_{b_i}$ is ${<}\kappa$-closed,
  so $\tilde p_i=\bigcup_{l<i} p'_l$ is in $P_{b_i}$ and, by induction,
  stronger than $p\restriction \sup_{l<i}b_l$. So we can
  extend $\tilde p_i$ to an element of $P_{b_i}$ stronger than $p\restriction b_i$.)
\end{proof}


\subsection{The Silver style iteration}

We now use the basic forcing $S(\kappa)$ in a reverse Easton iteration, the
first part acting as preparation to allow the preservation of measurability.
This method was developed by Silver to violate GCH at a measurable, and has
since been established as one of the basic tools in forcing with large
cardinals. We do not repeat all the details here, a more
detailed account can be
found in~\cite[21.4]{MR1940513}.
Note that here we do not just need to preserve measurability or
supercompactness (for this, we could just use Laver's general
result~\cite{MR0472529}), we need specific properties of the Silver iteration.

Fix a $j: V\to M$ such that 
\begin{equation}
  \text{$M$ is closed under $\kappa^{++}$-sequences.}
\end{equation}
In particular, $\cf(j(\kappa))>\kappa^+$.

We will use the reverse Easton iteration $(R_a,S(a))_{a\leq \kappa}$,
for $S(a)$ defined as above.
$R_{\kappa}$ is
the preparation that allows us to preserve measurability
(and we will not need it for anything else); we will
look at $R_\kappa\ast P_{a}$ for $a\leq \kappa^+$, and in particular at
$R_{\kappa+1}=R_\kappa\ast P_{\kappa^+}$ (recall that $S(\kappa)=
P_{\kappa^+}$).
We claim that $R_{\kappa+1}$ forces what we want.
We will also use $j(R_\kappa\ast P_{a})\in M$.
We get the usual properties:
\begin{itemize}
  \item The definition of $R$ is sufficiently absolute.
    In particular, we can (in $M$) factorize 
    $j(R_{\kappa+1})=R_{j(\kappa)+1}$ as
    $R_{\kappa+1}\ast R'$, where 
    $R'$ is the quotient forcing $R^{\kappa+1}_{j(\kappa)+1}$. 
    Note that $R'$ is ${<}\kappa^{+++}$-closed 
    (in $M$ and therefore in $V$ as well).
  \item Assume that $G$ is $R_{\kappa+1}$-generic over $V$ (and $M$).
    $M[G]$ is closed (as subset of $V[G]$) 
    under $\kappa^+$-sequences. In particular, $\kappa^+$
    is the same (and also equal to $2^\kappa$)
    in $V$, $V[G]$ and $M[G]$.
  \item For $p\in R_{\kappa+1}$, the domain of $j(p)$ is in
    $\kappa\cup \{j(\kappa)\}$, moreover
    $j(p)\restriction \kappa=p\restriction
    \kappa$ and
    $j(p)(j(\kappa))$ is isomorphic to $p(\kappa)$ such that
    $a\in\dom(p(\kappa))$ is mapped to $j(a)$.
    The image of $G$ under $j$ is element of
    $V[G]$ and subset of $M$ of size $\kappa^+$, therefore
    element of $M[G]$. 
    For $p\in G$ we can split in $M$ the
    condition $j(p)$ into
    $p\restriction \kappa$ (which is in $G$ anyway) and 
    $j(p(\kappa))$. We can assume that $G$ actually is 
    $R_\kappa\ast P'_{\kappa^+}$-generic (since 
    $P'_{\kappa^+}$ is dense in $P_{\kappa^+}$).
    Then $j(p(\kappa))$ is a $P'_{j(\kappa^+)}$-condition.
    So in $M[G]$,
    the set $\{j(p(\kappa)):\, p\in G\}$ is a directed 
    subset of $P'_{j(\kappa^+)}$ of size $\kappa^+$,
    therefore the union is a $P'_{j(\kappa^+)}$-condition $q_G$,
    a matrix of height $\kappa$ (which is less than $j(\kappa)$,
    so no contradiction to the definition of $P'_a$)
    and with domain $j''\kappa^+$
    (which has size $\kappa^+<j(\kappa)^{M[G]}$).
    We call this condition $q_G$, the minimal $G$-master condition.
  \item In $M[G]$, we call $q\in R'$ a $G$-master condition if it
    is stronger than $q_G$.
  \item If $H$ contains some $G$-master condition and is  $R'$-generic 
    over $V[G]$ (and therefore $M[G]$ as well), then
    we can extend in $V[G][H]$ the embedding $j$ 
    to $V[G]\to M[G][H]$
    by setting $j(\tau[G])=j(\tau)[G][H]$. This defines 
    in $V[G][H]$ 
    a normal ultrafilter $U=\{A[G]:\, \kappa\in j(A)[G][H]\}$ over $V[G]$.
    Since $R'$ is sufficiently closed, $U$ is already element of $V[G]$.
\end{itemize}

\begin{Def}
  In $V[G]$, $a\in \kappa^+$ is called a {\em positive index}, if 
  \begin{equation}
    (\forall \zeta<j(\kappa))\, (\exists q\ G\text{-master condition})\,
    q\forc (\kappa\in j(B_a)\ \&\ j(g_a)(\kappa)>\zeta).
  \end{equation}
  Otherwise, $a$ is called a {\em null-index}.
\end{Def}

Here we interpret $B_a$ and $g_a$ as $R_\kappa\ast P_a$-names in the canonical
way, so the $j$-images are $R_{j(\kappa)}\ast P_{j(a)}$-names.
In particular, whether $\kappa\in j(B_a)\ \&\ j(g_a)(\kappa)>\zeta$ holds is
already decided in the $R_{j(\kappa)}\ast P_{j(a)}$
extension, so we can assume that the $G$-master condition $q$ of the definition
only consists of the required minimal master condition $q_G$ ``from $j(a)$ onwards'',
more exactly we can assume:
\begin{itemize}
  \item $q\in R_{j(\kappa)+1}$ is factorized as $x \ast y$, for 
    $x\in R_{j(\kappa)}$ and $y$ is $R_{j(\kappa)}$-name for 
    a condition in $P'_{j(\kappa^+)}$.
  \item $x$ forces that $(j(b),\alpha)$ is not in the domain 
    of the matrix $y$ for any $b\geq a$ and $\alpha\geq \kappa$.
\end{itemize}

In particular, 
we can extend $q$ to a master condition $q'$ forcing that
\begin{equation}\label{eq:neu4}
  \text{$j(f_b)(\kappa)=0$ for all $b\geq a$.}
\end{equation}
Similarly, we can extend $q$ to a master condition $q'$ forcing that
\begin{equation}\label{eq:neu5}
  \text{$j(f_a)(\kappa)=\zeta$ and $j(f_b)(\kappa)=0$ for all $b>a$.}
\end{equation}

\begin{Lem}
  If $a$ is null and $b>_T a$, then $b$ is null as well. 
  Also, $0$ is a positive index.
\end{Lem}
\begin{proof}
  Pick $\zeta<j(\kappa)$
  such that every master condition forces $j(g_a)(\kappa)<\zeta$ or $\kappa\notin
  j(B_a)$. But the empty condition forces $j(g_b)(\kappa)\leq j(g_a)(\kappa)$ and
  $j(B_b)\subseteq j(B_a)$.
\end{proof}

\begin{Def}\label{def:I}
  In $V[G]$, we define the ideal $I_0$ by 
  $A\in I_0$ iff there is 
  an $X\subseteq \kappa^+$ of size $\kappa$ consisting of
  null-indices such that
  \begin{equation}
    A\subseteq \nabla_{i\in X}B_i\text{ modulo a club set.}%
  \footnote{%
  As mentioned above, $\nabla_{i\in X}B_i$ is only defined modulo a 
  club set, since $X$ is not canonically isomorphic to $\kappa$
  (it is just a subset of $\kappa^+$ of size $\kappa$).
  To avoid ambiguity, we just fix from now on for each such $X$ a bijection 
  to $\kappa$ and make $\nabla_{i\in X}B_i$ well defined;
  still we use ``subset modulo club set'' in the definition
  on $I_0$.%
  }
\end{equation}
\end{Def}

\begin{Lem}
$I_0$ is a normal ideal on $\kappa$
\end{Lem}
\begin{proof}
  Assume that $A_i\cap C_i\subseteq
  \nabla_{l\in X_i}B_l$ for all $i\in \kappa$. Then $(\nabla_{i\in\kappa} A_i)\cap \Delta_{i\in\kappa}
  C_i\subseteq \nabla_{l\in \bigcup X_i}B_l$ modulo a club set.
\end{proof}

By elementarity, if $q$ is a $G$-master condition and
if $\varphi(c,B_\alpha[G],g_\alpha[G])$ holds in $V[G]$ 
for some $c\in V$, then for all $H$ containing $q$ we get in $M[G][H]$
\begin{equation}
   \varphi(j(c),j(B_\alpha)[G][H],j(g_\alpha)[G][H]).
\end{equation}

\begin{Lem}\label{lem:jhwetet}
  In $V[G]$ the following holds:
  If $a$ is a positive index, then $A_a$ is $I_0$-positive.
\end{Lem}

Note that this implies: $a$ is a positive index iff $B_a$ is a $I_0$-positive
set; and $I_0$ is nontrivial (since $0$ is a positive index).

\begin{proof} Assume otherwise, and fix an
appropriate $X$ and a club set $C$, i.e., 
\begin{equation}\label{eq:gubbug}
  A_a\cap C\subseteq \nabla_{i\in X}B_i.
\end{equation}
Since $X$ consists of null-indices, there is for each $b\in X$ a 
$\zeta_b<j(\kappa)$ such that every master condition forces
$\kappa\notin j(B_b)$ or $j(g_b)(\kappa)<\zeta_b$.
Since $\cf(j(\kappa))>\kappa$, we can find an upper bound $\xi$ for
all $\zeta_b$. So every master condition forces
\begin{equation}\label{eq:blubb}
  \kappa\in j(B_b)\text{ implies }j(g_b)(\kappa)<\xi\text{ for all }b\in X.
\end{equation}
Since $a$ is a positive index, we can find a master condition $q$ 
forcing $\kappa\in j(B_a)$ and $j(g_a)(\kappa)>\xi+1$.
According to~\eqref{eq:neu5} we can extend $q$ to $q'$ such that 
\begin{equation}\label{eq:lipr}
  j(f_a)(\kappa)>\xi \quad \text{and}\quad
  j(f_c)(\kappa)=0\text{ for all }c>a.
\end{equation}
Since $C$ is club, $\kappa$ is forced to be in $j(C)$.
So $q'$ forces $\kappa\in j(A_a)\cap j(C)$.
According to~\eqref{eq:gubbug},
$A_a\cap C\subseteq \nabla_{b\in X}B_b$ holds in $V[G]$,
so $q'$ forces $\kappa\in j(A_a)\cap j(C)\subset j(\nabla_{b\in X}B_b)$.
Let $Z$ be the sequence $(B_b)_{b\in X}$.
Recall that we fixed (in $V[G]$) some bijection $i$ of
$\kappa$ to $X$, to make $\nabla Z$ well defined.
So $j(\nabla Z)$ uses $j(i)$, a bijection from
$j(\kappa)$ to $j(X)$; and $\kappa\in j(\nabla Z)$ means:
There is an $\alpha<\kappa$ such that $\kappa\in j(Z)_{j(i)(\alpha)}$.
Note that $j(Z)_{j(i)(\alpha)}=j(B_{i(\alpha)})$ and set $b=i(\alpha)\in X$.
So 
\begin{equation}\label{eq:eu3}
  \kappa\in  j(B_b)\text{ for  some }b\in X 
  \text{ (in particular, $b$ is null-index).}
\end{equation}
We further extend $q'$ to some $q''$ deciding the $b$ of~\eqref{eq:eu3}.
So $q''$ forces
\begin{equation}\label{eq:fnfz}
  \kappa\in j(A_a \cap B_b)\text{ for the null-index $b$}.
\end{equation}
We will get a contradiction by case distinction on the position of $b$
relative to $a$ in the tree $T$:
\begin{itemize}
  \item $b<_T a$: This contradicts the fact that $b$ is a null-index 
    and $a$ not. 
  \item $a\lhd_T b$: 
    Then $g_b=f_a$, and
    $q''$ forces that $\kappa\in j(B_b)$ and
    $j(g_b)(\kappa)\geq j(f_a)(\kappa)>\xi>\zeta_b$,
    contradicting~\eqref{eq:blubb}.
  \item $a\lhd_T c$ and $c<_T b$:
    Then $c$ is 
    (as an ordinal) bigger than $a$, and 
    $q''$ forces $\kappa\notin j(A_c)$.
    So $\kappa\notin j(B_b)\subseteq j(A_c)$.
  \item So $a$ and $b$ have to be incomparable in $T$, and there is some node
    $c$ where $a$ and $b$ split.
    Let  $a'$ and $b'$ the according immediate $T$-successors of $c$.
    So $a'\rhd_T c$, $b'\rhd_T c$, $a'\leq_T a$, $b'\leq_T b$ and $a'\neq b'$.
    Let $\underline m$ be the minimum of $a',b'$ (as ordinals)
    and $\overline m$ the maximum.
    According to~\eqref{eq:lublub} $A_{\underline{m}}\cap B_{\overline{m}}$
    is nonstationary, so $\kappa\notin j(A_{\underline{m}}\cap B_{\overline{m}})$.
    So~\eqref{eq:fnfz}  implies that $b'=b={\underline{m}}$.
    Also $j(g_b)(\kappa)=j(f_c)(\kappa)\geq j(f_a)(\kappa)>\xi$
    according to~\eqref{eq:lipr} which contradicts~\eqref{eq:blubb}.
\end{itemize}
\end{proof}

\begin{Lem}\label{lem:yrtu}
  In $V[G]$, empty has a winning strategy for $\BMzne(I_0)$.
\end{Lem}
\begin{proof}
Assume that we have a partial run of the game of length $n$, 
corresponding to the node $a$ in $T$,
and empty has played $X_n$ as last move, which is a subset of $A_a$.
Assume that nonempty plays the $I_0$-positive set $B^0\subseteq A_a$.
Let $b\rhd_T a$ be such that $X_{n+1}:=A_b\cap B^0$ is $I_0$-positive,
and let $X_{n+1}$ be empty's answer (and $b$ be the new $T$-node corresponding
to the new partial run). This is a winning strategy since 
$f_n(\alpha)$  decreases along
every branch of $T$. It remains to be shown that we can find a $b\rhd_T a$
as above: $B^0$ itself is enumerated as $B^0_c$ by the bookkeeping at some
stage $c\rhd_T a$.
Recall that $B_c=B^0_c\setminus \nabla_{d<c,d\rhd_T a}A_d$.
If $B_c$ is positive, then we can set $b=c$.
Otherwise, since $B^0_c$ is positive, some $B^0_c\cap A_d$ has to be positive
for some $d<c,d\rhd_T a$ (since $I_0$ is normal); and we can set $b=d$.
\end{proof}

It remains to be shown:
\begin{Lem}\label{lem:hjwte}
  In $V[G]$, for every $I_0$-positive $X$
  there is a normal ultrafilter $D_1$ extending the dual of $I_0$
  and containing $X$.
\end{Lem}
\begin{proof}
  It is enough to show: If  
  $Y$ is $I_0$-positive, then there is a master condition $q$
  forcing
  \begin{equation}
    \kappa\in j(Y)\text{ and } \kappa\notin j(B_b)\text{ for all null-indices
    }b.
  \end{equation}
  Let $X$ be the set of indices $a$ such that $Y\cap A_a$ is
  $I_0$-positive. Assume $a\in X$. We will use $Y\cap A_a$ as $B^0_b$ for
  some $b\rhd_T a$. 
  We have to distinguish two cases:
  \begin{description}
    \item[Case 1] There is a positive  $c\rhd_T a$ such that 
      $B_c\subseteq Y$. 
      In particular, this will be the case if $b$ itself is positive, i.e.\
      if $B_b=B^0_b\setminus \nabla_{c<b,c\rhd_T a} A_c$
      is $I_0$-positive.
    \item[Case 2] There is no such $c$.
      In particular, in this case $b$ is a null-index, so
      $Y\cap A_a$ is covered (modulo $I_0$) by $\nabla_{c<b,c\rhd_T a}A_c$.
      Then $c\notin X$ for any $c\geq b$ such that $c\rhd_T a$.
      So at most $\kappa$ many immediate $T$-successors of $a$ are in $X$;
      and $Y\cap A_a$ is covered (modulo $I_0$) by $\nabla_{c\rhd_T a,c\in
      X}A_c$
      as well.
  \end{description}
  We claim that Case 1 has to occur for some $a$.
  Otherwise,  $X$ is a subtree of $T$ such that every node has at most
  $\kappa$ many successors, i.e., there are only $\kappa$ many branches through
  $X$. By induction on $n$,
  $Y$ is covered (module $I_0$) by $\nabla_{c\in X,\ T\text{-height}(c)=n}A_c$.
  But for any branch $b$, the set $\bigcap_{n\in\omega} A_{b(n)}$ is empty
  (witnessed by the decreasing sequence $f_{b(n)}$), a contradiction.

  So we can pick a $T$-minimal $b$ such that Case 1 holds.
  Note that $|j''\kappa^+|<\cf(j(\kappa))$.
  For every null-index $c$ there is a witness $\xi_c<j(\kappa)$,
  so there is a universal bound $\xi$.
  Since $b$ is a positive index, we can find a master condition
  $q$ forcing $j(g_b)(\kappa)>\xi$ and $\kappa\in j(B_b)$. 
  Recall that $B_b\subseteq Y$ (mod $I_0$), so
  $q$ forces that $\kappa\in j(Y)$. We now 
  extend $q$ to $q'$ so that it forces $\kappa\notin j(A_c)$ for all $c>b$.
  Then $q'$ is as required: $\kappa\notin j(B_c)$ for any null-index $c$,
  by a similar case distinction as in the proof of Lemma~\ref{lem:jhwetet}.
\end{proof}

\bibliographystyle{amsplain}
\bibliography{morepressing}

\end{document}